\documentclass[english]{amsart}
\usepackage[T1]{fontenc}
\usepackage[latin9]{inputenc}
\usepackage{amstext}
\usepackage{amsthm}
\usepackage{amssymb}

\makeatletter
\numberwithin{equation}{section}
\numberwithin{figure}{section}
\theoremstyle{plain}
\newtheorem{thm}{\protect\theoremname}[section]
\theoremstyle{definition}
\newtheorem{defn}[thm]{\protect\definitionname}
\theoremstyle{definition}
\newtheorem{example}[thm]{\protect\examplename}
\theoremstyle{plain}
\newtheorem{fact}[thm]{\protect\factname}
\theoremstyle{plain}
\newtheorem{lem}[thm]{\protect\lemmaname}
\theoremstyle{plain}
\newtheorem{cor}[thm]{\protect\corollaryname}

\makeatother

\usepackage{babel}
\providecommand{\corollaryname}{Corollary}
\providecommand{\definitionname}{Definition}
\providecommand{\examplename}{Example}
\providecommand{\factname}{Fact}
\providecommand{\lemmaname}{Lemma}
\providecommand{\theoremname}{Theorem}

\begin{document}
\title{An Analogue to Infinitery Hales-Jewett Theorem}
\author{ANINDA CHAKRABORTY AND SAYAN GOSWAMI}
\email{anindachakraborty2@gmail.com}
\address{Government General Degree College at Chapra, Chapra.}
\email{sayan92m@gmail.com}
\address{Department of Mathematics, University of Kalyani, Kalyani.}
\begin{abstract}
In a recent work, N. Hindman, D. Strauss and L. Zamboni have shown
that the Hales-Jewett theorem can be combined with a sufficiently
well behaved homomorphisms. In this paper we will show that those
combined extensions can be made if we replace the alphabet by an increasing
sequence of alphabets, infact it holds for some Ramsey theoretic small
sets. To obtained this we achieved some interesting cofigurations. 
\end{abstract}

\maketitle

\section{Preliminaaries}

Let $\omega=\mathbb{N}\cup\left\{ 0\right\} $, where $\mathbb{N}$
is the set of positive integers. Then $\omega$ is the first infinite
ordinal. Let $\mathcal{P}_{f}\left(X\right)$ be the set of all finite
subsets of a set $X$.

Given a nonempty set $\mathbb{A}$ (or alphabet) we let $w\left(\mathbb{A}\right)$
be the set of all finite words $w=a_{1}a_{2}\ldots a_{n}$ with $n\geq1$
and $a_{i}\in\mathbb{A}$. The quantity $n$ is called the length
of $w$ and denoted $\left|w\right|$. The set $w\left(\mathbb{A}\right)$
is naturally semigroup under the operation of concatenation of words.
We will denote the empty word by $\theta.$ For each $u\in w\left(\mathbb{A}\right)$
and $a\in\mathbb{A}$, we let $\left|u\right|_{a}$ be the number
of occurrences of $a$ in $u$, we will identify the elements of $\mathbb{A}$
with the length-one words over $\mathbb{A}$.

Let $v$ (a variable) be a letter not belonging to $\mathbb{A}$.
By a variable word over $\mathbb{A}$ we mean a word $w$ over $\mathbb{A}\cup\left\{ v\right\} $
with $\left|w\right|_{v}\geq1$. We let $S_{1}\left(\mathbb{A}\right)$
be the set of one-variable words over $\mathbb{A}$. If $w\in S_{1}\left(\mathbb{A}\right)$
and $a\in\mathbb{A}$, then $w\left(a\right)\in w\left(\mathbb{A}\right)$
is the result of replacing each occurrence of $v$ by $a$.

A finite coloring of a set $A$ is a function from $A$ to a finite
set $\left\{ 1,2,\ldots,n\right\} $. A subset $B$ of $A$ is monochromatic
if the function is constant on $B$. If $\mathbb{A}$ be any finite
nonempty set and $S$ be the free semigroup of all words over the
alphabet $\mathbb{A}$, then the Hales-Jewett Theorem states that
for any finite coloring of the $S$ there is a variable word over
$\mathbb{A}$ all of whose instances are the same color.
\begin{thm}
\cite{key-3} Assume that $\mathbb{A}$ is finite. For each finite
coloring of $S_{0}$ there exists a variable word $w$ such that $\left\{ w\left(a\right):a\in\mathbb{A}\right\} $
is monochromatic.
\end{thm}

Now, we need to recall some definitions from \cite{key-6}.
\begin{defn}
Let $n\in\mathbb{N}$ and $v_{1},v_{2},\ldots,v_{n}$ be distinct
variables which are not members of $\mathbb{A}$.

(a) An $n$-variable word over $\mathbb{A}$ is a word $w$ over $\mathbb{A}\cup\left\{ v_{1},v_{2},\ldots,v_{n}\right\} $
such that $\left|w\right|_{v_{i}}\geq1$ for each $i\in\left\{ 1,2,\ldots,n\right\} $.

(b) If $w$ is an $n$-variable word over $\mathbb{A}$ and $\vec{x}=\left(x_{1},x_{2},\ldots,x_{n}\right)$,
then $w\left(\vec{x}\right)$ is the result of replacing each occurrence
of $v_{i}$ in $w$ by $x_{i}$ for each $i\in\left\{ 1,2,\ldots,n\right\} $.

(c) If $w$ is an $n$-variable word over $\mathbb{A}$ and $u=a_{1}a_{2}\ldots a_{n}$
is a length $n$ word, then $w\left(u\right)$ is the result of replacing
each occurrence of $v_{i}$ in $w$ by $a_{i}$ for each $i\in\left\{ 1,2,\ldots,n\right\} $. 
\end{defn}

This particular homomorphism is very useful to us.
\begin{defn}
Let $n\in\mathbb{N}$, for a set of alphabet $\mathbb{A}$, let $\vec{a}\in\mathbb{A}^{n}$.
Than the homomorphism $h_{\vec{a}}:S_{n}\left(\mathbb{A}\right)\rightarrow w\left(\mathbb{A}\right)$
is defined by $h_{\vec{a}}\left(w\right)=w\left(\vec{a}\right)$ for
all $w\in h_{\vec{a}}:S_{n}\left(\mathbb{A}\right)$. 
\end{defn}

\begin{defn}
Let $S,T\text{ and }R$ be semigroups (or partial semigroups) such
that $S\cup T$ is a semigroup (or partial semigroup) and $T$ is
an ideal of $S\cup T$. Then a homomorphism $\tau:T\rightarrow R$
is said to be $S$-independent if, for every $w\in T$ and every $u\in S$,
\[
\tau\left(uw\right)=\tau\left(w\right)=\tau\left(wu\right).
\]
\end{defn}

For an example, let $T$ be a semigroup with identity $e$. Then for
any $n\geq1$, a homomorphism $\tau:S_{n}\left(\mathbb{A}\right)\cup S_{0}\left(\mathbb{A}\right)\rightarrow T$
is $S_{0}\left(\mathbb{A}\right)$-independent if $\tau\left[S_{0}\left(\mathbb{A}\right)\right]=\left\{ e\right\} $.

The following is a version of a multi-variable extension of the Hales-Jewett
Theorem:
\begin{thm}
Assume that $\mathbb{A}$ is finite. Let $w\left(\mathbb{A}\right)$
be finitely colored and let $n\in\mathbb{N}$. There exists $w\in S_{n}\left(\mathbb{A}\right)$
such that $\left\{ w\left(\vec{x}\right):\vec{x}\in\mathbb{A}^{n}\right\} $
is monochromatic.
\end{thm}

Let, 
\[
\mathbb{A}_{1}\subseteq\mathbb{A}_{2}\subseteq\mathbb{A}_{3}\subseteq\ldots
\]
 is an increasing sequence of alphabets. Then for each $n,i\in\mathbb{N}$,
$S_{n}\left(\mathbb{A}_{i}\right)$ is the set of all $n$-variable
words over $\mathbb{A}_{i}$.

Let, $S_{n}=\bigcup_{i=1}^{\infty}S_{n}\left(\mathbb{A}_{i}\right)$
be the set of all $n$-variable words over $\mathbb{A}=\bigcup_{i=1}^{\infty}\mathbb{A}_{i}$
and $S_{0}=\bigcup_{i=1}^{\infty}w\left(\mathbb{A}_{i}\right)$ be
the set of all words of finite length over $\mathbb{A}=\bigcup_{i=1}^{\infty}\mathbb{A}_{i}$.

The following theorem is due to N.Karagiannis, which is a stronger
version of weak Carlson-Simpson theorem. 
\begin{thm}
\cite[Theorem 3]{key-7} Let $\left(\mathbb{A}_{i}\right)_{i=1}^{\infty}$
be an increasing sequence of finite alphabets and let $\mathbb{A}=\bigcup_{i=1}^{\infty}\mathbb{A}_{i}$.
Then for every finite coloring of $S_{0}$ there exists a sequence
$\left(w_{n}\left(x\right)\right)_{n=0}^{\infty}$ of variable words
over $\mathbb{A}$ such that for every $n\in\mathbb{N}$ and every
$m_{1}<m_{2}<\ldots<m_{n}$, the words of the form $w_{m_{0}}\left(a_{0}\right)w_{m_{1}}\left(a_{1}\right)\ldots w_{m_{n}}\left(a_{n}\right)$
with $a_{i}\in\mathbb{A}_{m_{i}}$, $i\in\left\{ i,2,\ldots,n\right\} $
are of the same color.
\end{thm}

In \cite{key-6}, authors have shown various combining extension of
Hales-Jewett theorem. After that N. Karagiannis has proved that weak
Carlson-Simpson theorem holds if one replaced the set of alphabet
$\mathbb{A}$ by an increasing chain of alphabets $\mathbb{A}_{1}\subseteq\mathbb{A}_{2}\subseteq\mathbb{A}_{3}\subseteq\ldots$. 

In this paper, we used the Stone-\v{C}ech compactification of a semigroup
(or partial semigroup) $S$ to develop some results from \cite{key-6}
in the infinite alphabet settings.

Let $\left(S,\cdot\right)$ be any discrete semigroup and denote its
Stone-\v{C}ech compactification as $\beta S$. $\beta S$ is the
set of all ultrafilters on $S$, where the points of $S$ are identified
with the principal ultrafilters. The basis for the topology is $\left\{ \bar{A}:A\subseteq S\right\} $,
where $\bar{A}=\left\{ p\in\beta S:A\in p\right\} $. The operation
of $S$ can be extended to $\beta S$ making $\left(\beta S,\cdot\right)$
a compact, right topological semigroup with $S$ contained in its
topological center. That is, for all $p\in\beta S$ the function $\rho_{p}:\beta S\rightarrow\beta S$
is continuous, where $\rho_{p}\left(q\right)=q\cdot p$ and for all
$x\in S$, the function $\lambda_{x}:\beta S\rightarrow\beta S$ is
continuous, where $\lambda_{x}\left(q\right)=x\cdot q$. For $p,q\in\beta S$
and $A\subseteq S$, $A\in p\cdot q$ if and only if $\left\{ x\in S:x^{-1}A\in q\right\} \in p$,
where $x^{-1}A=\left\{ y\in S:x\cdot y\in A\right\} $. 

Since $\beta S$ is a compact Hausdorff right topological semigroup,
it has a smallest two sided ideal denoted $K\left(\beta S\right)$,
which is the union of all of the minimal right ideals of $S$, as
well as the union of all of the minimal left ideals of $S$. Every
left ideal of $\beta S$ contains a minimal left ideal and every right
ideal of $\beta S$ contains a minimal right ideal. The intersection
of any minimal left ideal and any minimal right ideal is a group and
any two such groups are isomorphic. Any idempotent $p$ in $\beta S$
is said to be minimal if and only if $p\in K\left(\beta S\right)$.
A subset $A$ of $S$ is then central if and only if there is some
minimal idempotent $p$ such that $A\in p$. For more details, the
reader can see \cite{key-5}.

We need to use some elementary structure of partial semigroup. Here
we recall some definitions of our need from \cite{key-4}.
\begin{defn}
A partial semigroup is defined as a pair $\left(G,\ast\right)$ where
$\ast$ is an operation defined on a subset $X$ of $G\times G$ and
satisfies the statement that for all $x,y,z$ in $G$, $\left(x\ast y\right)\ast z=x\ast\left(y\ast z\right)$
in the sense that if either side is defined, so is the other and they
are equal.
\end{defn}

If $\left(G,\ast\right)$ is a partial semigroup, we will denote it
by $G$, when the operation $\ast$ is clear from the context. Now,
we give an example which will be useful in our work.
\begin{example}
Let us consider any sequence $\left\langle x_{n}\right\rangle _{n=1}^{\infty}$
in $\omega$ and let 
\[
G=\text{FS}\left(\left\langle x_{n}\right\rangle _{n=1}^{\infty}\right)=\left\{ \sum_{j\in H}x_{j}:H\in\mathcal{P}_{f}\left(\mathbb{N}\right)\right\} 
\]
 and 
\[
X=\left\{ \left(\sum_{j\in H_{1}}x_{j},\sum_{j\in H_{2}}x_{j}\right):H_{1}\cap H_{2}=\emptyset\right\} .
\]
 Define $\ast:X\rightarrow G$ by 
\[
\left(\sum_{j\in H_{1}}x_{j},\sum_{j\in H_{2}}x_{j}\right)\longrightarrow\sum_{j\in H_{1}}x_{j}+\sum_{j\in H_{2}}x_{j}.
\]
 It is easy to check that $G$ is a commutative partial semigroup.
One can similarly check the same for 
\[
G=\text{FP}\left(\left\langle x_{n}\right\rangle _{n=1}^{\infty}\right)=\left\{ \prod_{j\in H}x_{j}:H\in\mathcal{P}_{f}\left(\mathbb{N}\right)\right\} .
\]
\end{example}

\begin{defn}
Let $G$ be a partial semigroup. 

(a) For $g\in G$, $\varphi\left(g\right)=\left\{ h\in G:g\ast h\text{ is defined}\right\} $.

(b) For $H\in\mathcal{P}_{f}\left(G\right)$, $\sigma\left(H\right)=\bigcap_{h\in H}\varphi\left(h\right)$.

(c) $\left(G,\ast\right)$ is adequate if and only if $\sigma\left(H\right)\neq\emptyset$
for all $H\in\mathcal{P}_{f}\left(G\right)$.
\end{defn}

Now, for a semigroup $S$, let $^{\mathbb{N}}S$ be the set of all
sequences in $S$ and let
\[
\mathcal{J}_{m}=\left\{ t=\left(t_{1},t_{2},\ldots,t_{m}\right)\in\mathbb{N}^{m}:t_{1}<t_{2}<...<t_{m}\right\} .
\]

\begin{defn}
\label{J}Let $\left(S,\cdot\right)$ be a semigroup and $A\subseteq S$. 
\begin{enumerate}
\item Then $A$ is a $J$-set if and only if for each $F\in\mathcal{P}_{f}\left(^{\mathbb{N}}S\right)$
there exist $m\in\mathbb{N}$, $a=\left(a_{1},a_{2},\ldots,a_{m+1}\right)\in S^{m+1}$
and $t=\left(t_{1},t_{2},\ldots,t_{m}\right)\in\mathcal{J}_{m}$ such
that for each $f\in F$, 
\[
\left(\prod_{j=1}^{m}a_{j}\cdot f\left(t_{j}\right)\right)\cdot a_{m+1}\in A.
\]
\item $J\left(S\right)=\left\{ p\in\beta S:\left(\forall A\in p\right)\left(A\text{ is a }J\text{-set}\right)\right\} .$
\item A set $A\subseteq S$ is called $C$-set if it is an element of some
idempotent in $J\left(S\right)$.
\end{enumerate}
\end{defn}

Now, we need to recall the definition of $J$-set for partial semigroups.
First, let $\mathcal{F}$ be the set of all adequate sequences in
$S$.
\begin{defn}
Let $G$ be an adequate partial semigroup. Then a set $A\subseteq G$
is a $J$-set if and only if for all $F\in\mathcal{P}_{f}\left(\mathcal{F}\right)$
and all $L\in\mathcal{P}_{f}\left(S\right)$, there exist $m\in\mathbb{N}$,
$a=\left(a_{1},a_{2},\ldots,a_{m+1}\right)\in S^{m+1}$ and $t=\left(t_{1},t_{2},\ldots,t_{m}\right)\in\mathcal{J}_{m}$
such that for all $f\in F$, 
\[
\left(\prod_{i=1}^{m}a_{i}\ast f\left(t_{i}\right)\right)\ast a_{m+1}\in A\cap\sigma\left(L\right).
\]
\end{defn}

Then we have proved one of the main theorem from \cite[Theorem 17]{key-6}
for infinite sequence of alphabets.

\section{Our Results}

Throughout this section, we have developed various results from \cite{key-6}
using infinite sequence of alphabets. Some of the results also hold
in various small sets, like $C$-sets, $J$-sets etc. The following
fact is very useful to us:
\begin{fact}
Now, let us take the homomorphism $\tau:S_{0}\cup S_{1}\rightarrow\omega$
be defined by $\tau\left(w\right)=\left|w\right|_{v}$. So, $\tau\left(w\right)=0\text{ if and only if }w\in S_{0}$.
Let $D\subseteq S_{0}$ is a $J$-set and let $\left\langle x_{n}\right\rangle _{n=1}^{\infty}$
be a sequence in $\mathbb{N}$. Let $y_{1}=0,\,y_{n+1}=x_{n}$, for
all $n\in\mathbb{N}$. Then $A=\text{FS}\left(\left\langle y_{n}\right\rangle _{n=1}^{\infty}\right)$
is an adequate partial semigroup. As, $\tau$ is a homomorphism, one
can easily check that $\tau^{-1}\left[A\right]$ is also an adequate
partial semigroup. Now, let $T$ is an adequate partial semigroup
and $S$ be any partial subsemigroup of $T$. Let $F$ be a nonempty
finite set of partial semigroup homomorphisms from $T$ to $S$, which
are $S$-preserving. Then, proceeding similar to the proof of \cite[Lemma2.1]{key-2},
it can be easily be verified that for any $J$-set $D\subseteq S$,
$\bigcap_{\nu\in F}\nu^{-1}\left[D\right]$ is a $J$-set in $T$.
So, let $T=S_{0}\cup S_{1}$, $\tau^{-1}\left[A\right]$ is also a
$J$-set in $T$. Also, it is easy to check that $S_{0}$ is not a
$J$-set in $S_{n}\cup S_{0}$.
\end{fact}

The following result are a variation of Hales-Jewett theorem for piecewise
syndetic set:
\begin{lem}
\label{pws to J} Let $m,n\in\mathbb{N}$ and $F=\left\{ h_{\vec{a}}:\vec{a}\in\mathbb{A}_{m}^{n}\right\} $
be a finite nonempty set of homomorphism mappings from $T=S_{n}\cup S_{0}\longrightarrow S_{0}$
which are equal to the identity mapping on $S_{0}$. Let $D\subseteq S_{0}$
is a piecewise syndetic set in $S_{0}$, then $\bigcap_{\nu\in F}h_{\vec{a}}^{-1}\left[D\right]$
contains $w_{n}\left(\vec{a}\right)$, for all $\vec{a}\in\mathbb{A}_{m}^{n}$
and $w_{n}\in S_{n}\left(\mathbb{A}\right)$.
\end{lem}

\begin{proof}
Let $\in$$\mathbb{N}$, then from \cite[Corollary 6]{key-6}, $\bigcap_{h_{\vec{a}}^{-1}\in F}h_{\vec{a}}^{-1}\left[D\right]$
is piecewise syndetic in $T$. Now, $S_{0}$ is not piecewise syndetic
in $T$. So,$S_{n}\cap\bigcap_{h_{\vec{a}}\in F}h_{\vec{a}}^{-1}\left[D\right]$
is piecewise syndetic . Now, choose $N>m$ such that $w_{n}\in S_{n}\left(A_{N}\right)\cap\bigcap_{h_{\vec{a}}\in F}h_{\vec{a}}^{-1}\left[D\right]$.
So, $w_{n}\in\bigcap_{h_{\vec{a}}\in F}h_{\vec{a}}^{-1}\left[D\right]$
and $w_{n}\left(\vec{a}\right)\in D$ for all $\vec{a}\in A_{m}^{n}$.
\end{proof}
Proceeding as above for the $J$-sets, we get the following corollary:
\begin{cor}
Let $m,n\in\mathbb{N}$ and $F=\left\{ h_{\vec{a}}:\vec{a}\in\mathbb{A}_{m}^{n}\right\} $
be a finite nonempty set of homomorphism mappings from $T=S_{n}\cup S_{0}\longrightarrow S_{0}$
which are equal to the identity mapping on $S_{0}$. Let $D\subseteq S_{0}$
is a $J$-set in $S_{0}$, then $\bigcap_{h_{\vec{a}}\in F}h_{\vec{a}}^{-1}\left[D\right]$
contains $w_{n}\left(\vec{a}\right)$, for all $\vec{a}\in A_{m}^{n}$
and $w_{n}\in S_{n}\left(A\right)$.
\end{cor}

\begin{proof}
From \cite[Theorem 2.3]{key-2} and Lemma \ref{pws to J}, the result
follows.
\end{proof}
In the next theorem we have proved \cite[Theorem 3]{key-7} for central
sets:
\begin{thm}
\label{central}Let $D\subseteq S_{0}$ be central then there exists
a sequence $\left(w_{i}\left(x\right)\right)_{i=1}^{\infty}$ of variable
words over $A$ such that for every $n\in\mathbb{N}$ and every 
\[
m_{1}<m_{2}<\ldots<m_{l}
\]
 the words of the form 
\[
w_{m_{1}}\left(\vec{a_{1}}\right)w_{m_{2}}\left(\vec{a_{2}}\right)\ldots w_{m_{l}}\left(\vec{a_{l}}\right)\in D
\]
 with $a_{1}\in A_{m_{1}}^{n},a_{2}\in A_{m_{2}}^{n},\ldots,a_{l}\in A_{m_{l}}^{n}$.
\end{thm}

\begin{proof}
Let us choose a minimal idempotent $p$ such that $D\in p$. Then
$D^{*}\in p$, where $D^{*}=\left\{ x\in D:x^{-1}D\in p\right\} $.
Now, let $F_{1}=\left\{ h_{\vec{a_{1}}}:\,\vec{a_{1}}\in A_{1}^{n}\right\} $.
So, from Lemma \ref{pws to J}, there exists $w\in S_{n}\left(A_{1}\right)$
such that $w_{1}\in S_{n}\cap\bigcap_{\vec{a_{1}}\in A_{1}^{n}}h_{\vec{a_{1}}}^{-1}\left[D^{*}\right]\Rightarrow w_{1}\left(\vec{a_{1}}\right)\in D^{*}$
for all $\vec{a_{1}}\in A_{1}^{n}$.

So, $D^{*}\cap\bigcap_{\vec{a_{1}}\in A_{1}^{n}}w_{1}\left(\vec{a_{1}}\right)^{-1}\left[D^{*}\right]$.

Let, $p\in\mathbb{N}$ and suppose we have a sequence $\left(w_{i}\left(x\right)\right)_{i=1}^{p}$
such that 
\[
w_{m_{1}}\left(\vec{a_{1}}\right)\ldots w_{m_{l}}\left(\vec{a_{l}}\right)\in D
\]
 for all $m_{1}<m_{2}<\ldots<m_{l}\leq p$.

Now, choose 
\[
E=D^{*}\cap\bigcap_{m_{1}<m_{2}<\ldots<m_{l}\leq p}w_{m_{1}}\left(\vec{a_{1}}\right)\ldots w_{m_{l}}\left(\vec{a_{l}}\right)^{-1}\left[D^{*}\right]
\]
. Then, choose $w_{p+1}\in S_{n}\cap\bigcap_{\overrightarrow{a_{p+1}}\in A_{p+1}^{n}}h_{\overrightarrow{a_{p+1}}}^{-1}\left[E\right]$.

So, 
\[
w_{m_{1}}\left(\vec{a_{1}}\right)\ldots w_{m_{l}}\left(\vec{a_{l}}\right)\in D^{*}
\]
, where $m_{1}<m_{2}<\ldots<m_{l}\leq p+1$.

Hence by induction we can choose such infinite sequence $\left(w_{i}\left(x\right)\right)_{i=1}^{\infty}$.
\end{proof}
In fact, an easy combinatorial approach can be followed to prove the
following generalized result of the above theorem, using $C^{*}\in p$,
whenever $C\in p$.
\begin{cor}
Let $D\subseteq S_{0}$ be $C$-set then there exists a sequence $\left(w_{n}\left(x\right)\right)_{n=1}^{\infty}$
of variable words over $A$ such that for every $n\in\mathbb{N}$
and every 
\[
m_{1}<m_{2}<\ldots<m_{l}
\]
 the words of the form 
\[
w_{m_{1}}\left(\vec{a_{1}}\right)w_{m_{2}}\left(\vec{a_{2}}\right)\cdots w_{m_{l}}\left(\vec{a_{l}}\right)\in D
\]
 with $a_{1}\in A_{m_{1}}^{n},a_{2}\in A_{m_{2}}^{n},\ldots,a_{l}\in A_{m_{l}}^{n}$.
\end{cor}

\begin{proof}
Same as in Theorem \ref{central}.
\end{proof}
Note that, for a $J$-set $D\subseteq S_{0}$ and a sequence $\left\langle x_{n}\right\rangle _{n=1}^{\infty}$
in $\mathbb{N}$, let $y_{1}=0$ and $y_{n+1}=x_{n}$ and $A=\text{FS}\left(\left\langle y_{n}\right\rangle _{n=1}^{\infty}\right)$
is an adequate partial semigroup.
\begin{thm}
\label{Th 4}Let $\tau:T=S_{0}\cup S_{1}\rightarrow\omega$ by $\tau\left(w\right)=\left|w\right|_{v_{1}}$.
Let, $D\subseteq S_{0}$ be a $J$-set and $\left\langle x_{n}\right\rangle _{n=1}^{\infty}$
be a sequence in $\mathbb{N}$.Then, there exists $\left\langle w_{n}\right\rangle _{n=1}^{\infty}$
such that 
\[
\left\{ w_{n}\left(a_{n}\right):a_{n}\in A_{n}\right\} \subseteq D\ \forall n\in\mathbb{N}
\]
 and for each $G\in\mathcal{P}_{f}\left(\mathbb{N}\right)$, 
\[
\sum_{n\in G}\tau\left(w_{n}\right)\in\text{FS}\left(\left\langle x_{n}\right\rangle _{n=1}^{\infty}\right).
\]
\end{thm}

\begin{proof}
Let $F_{1}=\left\{ h_{a_{1}}:a_{1}\in A_{1}\right\} $ be a finite
set of partial semigroup homomorphism from $\tau^{-1}\left(A_{1}\right)$
to $S_{0}$ where $A_{1}=\text{FS}\left(\left\langle y_{n}\right\rangle _{n=1}^{\infty}\right)$.

Let $D\subseteq S_{0}$ is a $J$-set then one can easily verify that
$\bigcap_{a_{1}\in A_{1}}h_{a_{1}}^{-1}\left[D\right]$ is a $J$-set
in $\tau^{-1}\left(A_{1}\right)$. As, $S_{0}$ is not a $J$-set
in $\tau^{-1}\left(A_{1}\right)$, there exists $w_{1}\in S_{1}\left(A_{m}\right)\cap\tau^{-1}\left(A_{1}\right)$
for some $m\geq1$such that $w_{1}\in\bigcap_{a_{1}\in A_{1}}h_{a_{1}}^{-1}\left[D\right]$.

So, $\tau\left(w_{1}\right)\in\text{FS}\left(\left\langle x_{n}\right\rangle _{n=1}^{\infty}\right)\subseteq A_{1}=\text{FS}\left(\left\langle y_{n}\right\rangle _{n=1}^{\infty}\right)$
and $w_{1}\left(a_{1}\right)\in D$\ $\forall a_{1}\in A_{1}$. Now
by induction, suppose we have proved that for some $p\in\mathbb{N}$,
$\left\{ w_{n}\left(a_{n}\right):a_{n}\in A_{n}\right\} \subseteq D$,
for all $n\in\left[1,p\right]$ and for $G\subseteq\left\{ 1,2,\ldots,p\right\} $,
$\sum_{n\in G}\tau\left(w_{n}\right)\in\text{FS}\left(\left\langle x_{n}\right\rangle _{n=1}^{\infty}\right)$.
Now let us take a sequence $\left\langle z_{n}\right\rangle _{n=1}^{\infty}$
such that $z_{1}=0$ and 
\[
A_{p+1}=\text{FS}\left(\left\langle z_{n}\right\rangle _{n=1}^{\infty}\right)\subseteq\bigcap_{x\in\left\{ \sum_{n\in G}\tau\left(w_{n}\right):G\subseteq\left\{ 1,2,\ldots,p\right\} \right\} }\left(-x+\text{FS}\left(\left\langle x_{n}\right\rangle _{n=1}^{\infty}\right)\right).
\]
 Then, $\tau^{-1}\left[A_{p+1}\right]$ is adequate partial semigroup
and as above there exists $m\geq p+1$ and a $w_{p+1}\in S_{1}\left(A_{m}\right)\cap\tau^{-1}\left[A_{p+1}\right]$
such that $w\in\bigcap_{a_{p+1}\in A_{p+1}}h_{a_{p+1}}^{-1}\left[D\right]$.

So, $\tau\left(w_{p+1}\right)\in A_{p+1}$ and $\left\{ w\left(a_{p+1}\right):a_{p+1}\in A_{p+1}\right\} $.
This proves the result.
\end{proof}
The following corollary is one of the interesting result:
\begin{cor}
Let $\tau:T=S_{0}\cup S_{1}\rightarrow\omega$ by $\tau\left(w\right)=\left|w\right|_{v_{1}}$.
Let, $D\subseteq S_{0}$ be a $C$-set and $\left\langle x_{n}\right\rangle _{n=1}^{\infty}$
be a sequence in $\mathbb{N}$. Then, there exists $\left\langle w_{n}\right\rangle _{n=1}^{\infty}$
and for any $\left\{ n_{1},n_{2},\ldots,n_{i}\right\} =G\in\mathcal{P}_{f}\left(\mathbb{N}\right)$,
such that 
\[
\left\{ w_{n_{1}}\left(a_{n_{1}}\right)w_{n_{2}}\left(a_{n_{2}}\right)\ldots w_{n_{l}}\left(a_{n_{l}}\right):a_{n_{i}}\in A_{n_{i}},\text{ where }n_{i}\in G\right\} \subseteq D
\]
 and $\sum_{n_{i}\in G}\tau\left(w_{n_{i}}\right)\in\text{FS}\left(\left\langle x_{n}\right\rangle _{n=1}^{\infty}\right).$
\end{cor}

\begin{proof}
As $D\subseteq S_{0}$ is a $C$-set, let us choose an idempotent
$p\in\delta\left(S_{0}\right)$ such that $D\in p$. Then, $D^{*}\in p$.
Now, proceeding as the above proof, assume, 
\[
E=D^{*}\cap\bigcap_{G=\left\{ n_{1},n_{2},\ldots,n_{i}\right\} \subseteq\left[1,p\right]}w_{n_{1}}\left(a_{n_{1}}\right)w_{n_{2}}\left(a_{n_{2}}\right)\ldots w_{n_{l}}\left(a_{n_{l}}\right)^{-1}\left[D^{*}\right]
\]
 and $A_{p+1}$ as similar as above.

Then, as $\bigcap_{a_{p+1}\in A_{p+1}}h_{a_{p+1}}^{-1}\left[E\right]$
is a $J$-set, we have the result. 
\end{proof}
In \cite[Theorem 2.5]{key-2}, authors proved \cite[Corollary 16]{key-6}
for $J$-sets. In the following theorem, we developed \cite[Theorem 2.5]{key-2}
for infinite sequence of alphabets.
\begin{thm}
Let $k,n\in\mathbb{N}$ with $k<n$ and let $T$ be the set of words
over $\left\{ v_{1},v_{2},\ldots,v_{k}\right\} $ in which $v_{i}$
occurs for each $i\in\left\{ 1,2,\ldots,k\right\} $. Given, $w\in S_{n}$,
let $\tau\left(w\right)$ be obtained from $w$ by deleting all occurrence
of elements of $A$ as well as occurrences of $v_{i}$, $k<i\leq n$.
Let $\left\langle y_{t}\right\rangle _{t=1}^{\infty}$ be a sequence
in $T$ and let $D\subseteq S_{0}$ is a $J$-set. Then there exists
a sequence of variable words $\left\langle w_{t}\right\rangle _{t=1}^{\infty}$
over $A$ such that $\left\{ w_{m}\left(a_{m}\right):a_{m}\in A_{m}^{n}\right\} \subseteq D$
for all $m\in\mathbb{N}$ and for every $G\in\mathcal{P}_{f}\left(\mathbb{N}\right)$,
\[
\prod_{i=1}^{i}\tau\left(w_{m_{i}}\right)\in\text{FP}\left(\left\langle y_{t}\right\rangle _{t=1}^{\infty}\right),
\]
 where $G=\left\{ m_{1}<m_{2}<\ldots<m_{l}\right\} $.
\end{thm}

\begin{proof}
We will proceed same as in Theorem \ref{Th 4}.

Let $T^{*}=T\cup\left\{ \theta\right\} $ and $\tau^{*}:S_{n}\cup S_{0}\rightarrow T^{*}$
be defined by 
\[
\tau^{*}\left(w\right)=\left\{ \begin{array}{c}
\tau\left(w\right),\text{if }w\in S_{n}\\
\theta,\text{ if }w\in S_{0}
\end{array}\right..
\]

Then $\left(\tau^{*}\right)^{-1}\left[\text{FP}\left(\left\langle y_{t}\right\rangle _{t=1}^{\infty}\right)\cup\theta\right]$
is a partial semigroup in $S_{n}$.

Let $A_{1}=\left(\tau^{*}\right)^{-1}\left[\text{FP}\left(\left\langle y_{t}\right\rangle _{t=1}^{\infty}\right)\cup\theta\right]$
and $F_{1}=\left\{ h_{\vec{a_{1}}}:\vec{a_{1}}\in A_{1}^{n}\right\} $. 

Then, $\bigcap_{\vec{a_{1}}\in A_{1}^{n}}h_{\vec{a_{1}}}^{-1}\left[D\right]$
is a $J$-set in $S_{n}\cup S_{0}$ and so in $S_{n}$ as $S_{0}$
is not a $J$-set in $S_{n}\cup S_{0}$.

So, there exists $w_{1}\in S_{n}\cap\bigcap_{\vec{a_{1}}\in A_{1}^{n}}h_{\vec{a_{1}}}^{-1}\left[D\right]$
such that $\tau\left(w_{1}\right)\in\text{FP}\left(\left\langle y_{t}\right\rangle _{t=1}^{\infty}\right)$.
i.e; $w_{1}\left(\vec{a_{1}}\right)\in D$ for all $\vec{a_{1}}\in A_{1}^{n}$
and $\tau\left(w_{1}\right)\in\text{FP}\left(\left\langle y_{t}\right\rangle _{t=1}^{\infty}\right)$.

from here proceeding as Theorem \ref{Th 4}, we get the result.
\end{proof}
Similarly, we get the following corollary.
\begin{cor}
Let $k,n\in\mathbb{N}$ with $k<n$ and let $T$ be the set of words
over $\left\{ v_{1},v_{2},\ldots,v_{k}\right\} $ in which $v_{i}$
occurs for each $i\in\left\{ 1,2,\ldots,k\right\} $. Given, $w\in S_{n}$,
let $\tau\left(w\right)$ be obtained from $w$ by deleting all occurrence
of elements of $A$ as well as occurrences of $v_{i}$, $k<i\leq n$.
Let $\left\langle y_{t}\right\rangle _{t=1}^{\infty}$ be a sequence
in $T$ and let $C\subseteq S_{0}$ is a $C$-set. Then there exists
a sequence of variable words $\left\langle w_{t}\right\rangle _{t=1}^{\infty}$
over $A$ such that $\left\{ w_{m}\left(a_{m}\right):a_{m}\in A_{m}^{n}\right\} \subseteq C$
for all $m\in\mathbb{N}$ and for every $G\in\mathcal{P}_{f}\left(\mathbb{N}\right)$,
\[
\prod_{i=1}^{l}\tau\left(w_{m_{i}}\right)\in\text{FP}\left(\left\langle y_{t}\right\rangle _{t=1}^{\infty}\right),
\]
 where $G=\left\{ m_{1}<m_{2}<\ldots<m_{l}\right\} $.
\end{cor}

\begin{proof}
Proceeding same as in Theorem \ref{Th 4}, we get the result.
\end{proof}
Similar to \cite[Lemma 17]{key-6}, in Theorem \ref{lem 17}, the
semigroup $T$ and the matrix $M$ satisfy all the appropriate hypotheses
to matrix multiplication makes sense and distributive.
\begin{thm}
\label{lem 17} Let $\left(T,+\right)$ be a commutative semigroup
with identity element $0$. Let $k,m,n\in\mathbb{N}$, and $M$ be
a $k\times m$ matrix. For $i\in\left\{ 1,2,\ldots,m\right\} $, let
$\tau_{i}$ be an $S_{0}$- independent homomorphism from $S_{n}$
to $T$. Define a function $\psi$ on $S_{n}$ by
\[
\psi\left(w\right)=\left(\begin{array}{c}
\tau_{1}\left(w\right)\\
\tau_{2}\left(w\right)\\
\vdots\\
\tau_{m}\left(w\right)
\end{array}\right),
\]
 with the property that for any collection of IP-sets $\left\{ C_{i}:i\in\left\{ 1,2,\ldots,k\right\} \right\} $
in $T$, there exists $a\in S_{n}$ such that $M\psi\left(a\right)\in\times_{i=1}^{k}C_{i}$.
Let $B_{i}=\text{FS}\left(\left\langle x_{n}^{\left(i\right)}\right\rangle _{n=1}^{\infty}\right)$
for $1\leq i\leq k$ be $k$ IP sets in $T$ and $D\subseteq S_{0}$
is a $C$-set in $S_{0}$, then, there exists $\left\langle w_{i}\right\rangle _{i=1}^{\infty}\subseteq S_{n}$
such that for each $\left\{ m_{1},m_{2},\ldots,m_{l}\right\} =G\in\mathcal{P}_{f}\left(\mathbb{N}\right)$,
we have 
\[
w_{m_{1}}\left(\overrightarrow{a_{m_{1}}}\right)w_{m_{2}}\left(\overrightarrow{a_{m_{2}}}\right)\cdots w_{m_{l}}\left(\overrightarrow{a_{m_{l}}}\right)\in D
\]
 and 
\[
M\circ\psi\left(w_{m_{1}}w_{m_{2}}\ldots w_{m_{l}}\right)\in\times_{i=1}^{k}B_{i}
\]
.
\end{thm}

\begin{proof}
Let $\phi:S_{n}\cup S_{0}\rightarrow B=\times_{i=1}^{k}\left(B_{i}\cup\left\{ 0\right\} \right)$,
where $\phi\left(w\right)=M\psi\left(w\right)$, so as in \cite[Lemma 17]{key-6},
$\phi$ is a homomorphism. Now, let us choose $p\in\beta S_{0}$ such
that $D\in p$ so then $D^{*}\in p$.

Now, $h_{\vec{a_{1}}}:\phi^{-1}\left[B\right]\rightarrow S_{0}$,
$\vec{a_{1}}\in A_{1}^{n}$ is a collection of partial semigroup homomorphism
fixed on $S_{0}$.

Now, $\bigcap_{\vec{a_{1}}\in A_{1}^{n}}h_{\vec{a_{1}}}^{-1}\left[D^{*}\right]\setminus S_{0}$
is a $J$-set in $\phi^{-1}\left[B\right]$ and so there exists $w_{1}\in\phi^{-1}\left[B\right]\setminus S_{0}$
such that $w_{1}\left(\vec{a_{1}}\right)\in D^{*}\forall\vec{a_{1}}\in A_{1}^{n}$
and $\phi\left(w_{1}\right)=M\psi\left(w_{1}\right)\in B$.

Now, take an IP set $B_{1}\subseteq-M\psi\left(w_{1}\right)+B$ and
clearly $\left(0,0,\ldots,0\right)\in B_{1}$. So, $S_{0}\subseteq\phi^{-1}\left[B_{1}\right]$.
Now, $\phi^{-1}\left(B_{1}\right)$ is again a partial semigroup,
and hence 
\[
h_{\vec{a_{2}}}:\phi^{-1}\left[B_{1}\right]\rightarrow E=D^{*}\cap\bigcap_{\vec{a_{1}}\in A_{1}^{n}}w_{1}\left(\vec{a_{1}}\right)^{-1}D^{*},\ \vec{a_{2}}\in A_{2}^{n}
\]
 is again a partial semigroup homomorphism fixed on $S_{0}$.

So, there exists $w_{2}\in\bigcap_{\vec{a_{2}}\in A_{2}^{n}}h_{\vec{a_{2}}}^{-1}\left[E\right]\cap\left(\phi^{-1}\left[B_{1}\right]\setminus S_{0}\right)$.
Thus, 
\[
h_{\vec{a_{2}}}\left(w_{2}\right)\in E=D^{*}\cap\bigcap_{\vec{a_{1}}\in A_{1}^{n}}w_{1}\left(\vec{a_{1}}\right)^{-1}D^{*}
\]
 for all $\vec{a_{2}}\in A_{2}^{n}$. Thus, $w_{2}\left(\vec{a_{2}}\right)\in D^{*}$,
$w_{1}\left(\vec{a_{1}}\right)=w_{2}\left(\vec{a_{2}}\right)\in D^{*}$
and $\phi\left(w_{2}\right)=M\psi\left(w_{2}\right)\in B_{1}$.

So, we have, 
\[
\left.\begin{array}{c}
w_{1}\left(\vec{a_{1}}\right)\in D^{*}\text{ for all }\vec{a_{1}}\in A_{1}^{n}\\
M\psi\left(w_{1}\right)\in B
\end{array}\right\} ----\left(\text{1}\right),
\]
\[
\left.\begin{array}{c}
w_{2}\left(\vec{a_{2}}\right)\in D^{*}\text{ for all }\vec{a_{2}}\in A_{1}^{n}\\
M\psi\left(w_{2}\right)\in B_{1}\subseteq-M\psi\left(w_{1}\right)+B
\end{array}\right\} ----\left(\text{2}\right),
\]

\[
\left.\begin{array}{c}
\begin{array}{c}
w_{1}\left(\vec{a_{1}}\right)w_{2}\left(\vec{a_{2}}\right)\in D^{*}\text{ for all }\vec{a_{1}}\in A_{1}^{n}\text{ and }\vec{a_{2}}\in A_{2}^{n}\\
M\psi\left(w_{2}\right)\in-M\psi\left(w_{1}\right)+B
\end{array}\end{array}\right\} ----\left(\text{3}\right)
\]
 and so $M\circ\psi\left(w_{1}w_{2}\right)=M\left(\psi\left(w_{1}\right)+\psi\left(w_{2}\right)\right)\in B$.

Similarly, by proceeding further with this iteration, we will get
the result.
\end{proof}
\textbf{Acknowledgment:} The second author of the paper acknowledges
the grant UGC-NET SRF fellowship with id no. 421333 of CSIR-UGC NET
December 2016. We would like to thank Prof. Dibyendu De for his helpful
comments on this paper.


\begin{thebibliography}{HSZ}
\bibitem[C]{key-1} T. Carlson,Some unifying principles in Ramsey
theory, Adv. in Math. 53 (1984) 265-290. 

\bibitem[CG]{key-2} A. Chakraborty \& S. Goswami, Hales-Jewett type
configurations in small sets, arXiv:2009.00772 {[}math.CO{]}.

\bibitem[FK]{key-1}H. Furstenberg and Y. Katznelson (1989). Idempotents
in compact semigroups and Ramsey Theory. Israel J. Math. 68, 257--270.

\bibitem[HJ]{key-3} A.W.Hales and R.I.Jewett, Regularity and positional
games, Trans. Amer. Math. Soc. 106 (1963), 222-229.

\bibitem[HP]{key-4} N. Hindman and K. Pleasant, Central sets theorem
for arbitrary adequate partial semigroups, Topology Proceedings, to
appear.

\bibitem[HS]{key-5} N. Hindman and D. Strauss, Algebra in the Stone-\v{C}ech
Compactification: Theory and Applications, second edition, de Gruyter,
Berlin, 2012.

\bibitem[HSZ]{key-6} N. Hindman, D. Strauss and L. Q. Zamboni, Combining
extensions of the Hales-Jewett Theorem with Ramsey Theory in other
structures, The Electronic Journal of Combinatorics 26(4) (2019),
\#P4.23.

\bibitem[K]{key-7} Nikolas Karagiannis, A combinatorial proof of
an infinite version of the Hales-Jewett theorem, Journal of combinatorics,Volume
4, Number 2, 273-291, 2013.
\end{thebibliography}
\end{document}